\documentclass[12pt,letterpaper]{article}
\usepackage{amsmath,amssymb,amsthm,amscd,tikz,tikz-cd,abbrviro,comment,array}
\graphicspath{ {./figs/} }
\def\fac#1#2{{#1/\lower0.5ex\hbox{\small$#2$}}}
\title{Complex slices on a real variety}
\author{Oleg Viro}
\date{}
\begin{document} 

\maketitle

\section{Transverse or not transverse}\label{s1}
\subsection{Transversality.}\label{s1.1}
Remind that subspaces $V$ and $W$ of a vector space $U$ are said to be {\em
transverse\/} if $V\cup W$ generates $U$, i.e., $U=V+W$. 
In a finite-dimensional $U$,
transversality is equivalent to a single condition in terms of dimensions:
$$V,W \text{ are transverse in }U \ \iff \ \dim U-(\dim V+\dim W)+\dim (V\cap W)=0$$

Two submanifolds of a smooth manifold are said to be {\em transverse\/} if
at each of their intersection points their tangent spaces are transverse.

Transversality of submanifolds is a condition of general position: 
any two submanifolds can be made transverse by a suffitiently small
perturbation of any of them and if submanifold are transverse then a
sufficiently small perturbation cannot destroy transversality. 

However, non-transverse intersections in Real Algebraic Geometry emerge 
inevitably if one of the submanifolds is a complex variety and the other
is the real point set of the ambient real algebraic variety. 

\subsection{Example: a plane curve.}\label{s1.2}  In $U=\C P^2$, the 
{\em complex\/} 
points set $V=B_\C$ of a {\em real\/} algebraic curve $B$ is {\em not 
transverse} 
to $W=\R P^2$. The intersection $V\cap W=V\cap\R P^2$ is the set $B_\R$
of real points  of $B$.

\centerline{$\vcenter{\hbox{\begin{picture}(0,0)%
\includegraphics{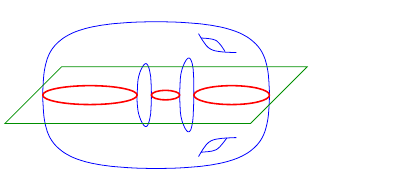}%
\end{picture}%
\setlength{\unitlength}{3315sp}%
\begin{picture}(3827,1622)(-405,-782)
\put(2071,-691){\makebox(0,0)[lb]{\smash{\fontsize{10}{12}\usefont{T1}{ptm}{m}{n}{\color[rgb]{0,0,1}$V=B_{\C}$}%
}}}
\put(2341,-106){\makebox(0,0)[lb]{\smash{\fontsize{10}{12}\usefont{T1}{ptm}{m}{n}{\color[rgb]{1,0,0}$V\cap W=B_\R$}%
}}}
\put(2566,209){\makebox(0,0)[lb]{\smash{\fontsize{10}{12}\usefont{T1}{ptm}{m}{n}{\color[rgb]{0,.56,0}$W=\R P^2$}%
}}}
\end{picture}%
}}$}\medskip

On the picture, $V$ and $W$ look transversal. However they lie not in
$\R^3$,
but in the 4-dimensional $\C P^2$, where intersection of transverse 
surfaces must be 0-dimensional.

The intersection of $V$ and $W$ can be 
made transversal  by an 
arbitrarily small
perturbation of $V$ 
under which $V$ stay in the class of 
{\bf complex} algebraic curves,
but  may not be a {\bf complexification of real} algebraic
curve.

\subsection{Real algebraic varieties in a nutshell.}\label{s1.3}
A real algebraic variety $X$ has\vspace{-9pt}
\begin{itemize}\setlength\itemsep{-4pt}
\setlength\topsep{0pt}
\setlength\parsep{0pt}
\item the set $\cx$ of complex points,
\item anti-holomorphic involution $\conj:\cx\to\cx$, and 
\item the set $\rx=\fix(\conj)$ of real points.  
\end{itemize}\vspace{-7pt}
If $X$ is nonsingular, then \vspace{-7pt}
\begin{itemize}\setlength\itemsep{-4pt}
\setlength\topsep{0pt}
\setlength\parsep{0pt}
\item $\cx$ is a smooth manifold with complex
structure, 
\item $\conj$ is diffeomorphism and 
\item $\rx$ is a smooth submanifold of
$\cx$.
\end{itemize}\vspace{-12pt}

\subsection{A complexification a real variety is never transverse to real
part of the ambient space.}\label{s1.4}
\begin{Th}\label{Th1}
If $Y$ is a real nonsingular subvariety of a real nonsingular variety  $X$, 
then both $\cy$ and $\rx$ are smooth submanifolds of $\cx$ and 
$\ry=\cy\cap\rx$. Then  $\cy$ and $\rx$ are not 
transverse in $\cx$ unless $\ry\cap\rx=\empt$ or $\dim Y=\dim X$.
\end{Th} 

\begin{proof} Assume that $\cy$ and $\rx$ are transverse. Then by
the transversality condition $$\dim_\R\cx-(\dim_\R\rx+\dim_\R\cy)+\dim_\R(\cy\cap\rx)=0.$$ 
Since \vspace{-7pt}
\begin{itemize}\setlength\itemsep{-4pt}
\setlength\topsep{-4pt}
\setlength\parsep{-4pt}
\item $\dim_\R\cx=2\dim_\R\rx$, 
\item $\dim_\R\cy=2\dim_\R\ry$ and 
\item $\cy\cap\rx=\ry$,
\end{itemize}\vspace{-7pt}
it follows\vspace{-7pt}
\begin{multline}
2\dim_\R\rx-(\dim_\R\rx+\dim_\R\cy)+\dim_\R\ry\\=\dim_\R\rx-\dim_\R\ry=0. 
\end{multline}
\end{proof}

In this paper we study how a {\em complex} subvariety $V$ 
of $\cx$ can meet $\rx$. 

\subsection{A complex cut of a real algebraic variety is real
algebraic.}\label{s1.5}
\begin{Th}\label{Th2} 
For any complex subvariety $V\subset\cx$, the intersection $V\cap\rx$ is the 
set  of real points of a real algebraic subvariety of $X$.
\end{Th}
\begin{proof}
 The complex conjugation involution
$\conj:\cx\to\cx$ maps $V$ to a complex subvariety. The
intersection $V\cap \conj V$ is a complex variety. Denote it by $Y$.
This is a real algebraic variety, because it is invariant under $\conj$.
Indeed, $\conj Y=\conj(V\cap\conj V)=\conj V\cap \conj^2V=(\conj V)\cap V= Y$. 

Further, $\ry=Y\cap\rx=(V\cap\conj V)\cap\rx=(V\cap\rx)\cap(\conj
V\cap\rx)=V\cap\rx$.
Thus, $V\cap\rx=Y\cap\rx=\ry$.
\end{proof}

\begin{cor} For any complex subvariety $V\subset \cx$ transversal to $\rx$, 
there exists a real algebraic subvariety $Y$ of $X$ such that $\cy\subset V$
and $\ry=V\cap\cx$. 
\end{cor}

\subsection{Complex transverse slices.}\label{s1.6}
Let $X$ be a nonsingular real algebraic variety.
Usually, $\cx$ contains lots of complex subvarieties.
They cut real algebraic varieties on $\rx$.
Most of the complex varieties are transverse to $\rx$.

Thus we have a large class of real algebraic varieties, which are 
transverse intersections of complex varieties with $\rx$.

Let $Y$ be a real algebraic subvariety of  a nonsingular real 
algebraic variety $X$, let $V$ be  a complex subvariety 
of $\cx$ transverse to $\rx$. If $\ry=V\cap\rx$,
then $Y$ is called a {\em slice} of $X$. 
We say also that $Y$ is {\em sliced\/} in $X$ by $V$. 

Sliced subvarieties are quite special. In particular:\vspace{-7pt}
\begin{itemize}\setlength\itemsep{-4pt}
\setlength\topsep{0pt}
\setlength\parsep{0pt}
\item  A sliced subvariety has {\em even codimension}. 
\item The {\em transverse\/} (aka normal) {\em bundle\/} of a real sliced
subvariety admits a {\em complex\/} structure.
\item In particular, the transverse bundle is {\em orientable}.
\item A choice of the slicing complex subvariety determines a {\em complex
structure\/} and {\em orientation\/} of the transverse bundle. 
\item With this orientation, a real {\em sliced\/} 
subvariety realizes in the ambient real variety an {\em
integer cohomology class\/} of the complementary dimension.
\item The orientation and the cohomology class are involved in topological
restrictions on mutual position of a slice with  other real algebraic
variaties. 
\item An irreducible {\em sliced\/} subvariety of codimension two is a {\em base of a real 
pencil\/} of hypersurfaces.
\item Conversely, a non-singular base of a real algebraic pencil of 
projective hypersurfaces is a slice.
\end{itemize}
To the best of my knowledge, real sliced subvarieties have not been 
discussed 
in literature. Below we consider their basic properties and study some 
topological theorems in which they are involved. 

Our considerations are 
restricted to non-singular varieties, although they admit generalizations
to singular varieties. 

\section{Codimension two slices and pencils of
hypersurfaces.}\label{s2}
The simplest class of sliced varieties are bases of generic real pencils of
projective hypersurfaces. 

\begin{Th}\label{ThPen}
A real algebraic projective variety of codimension 2 is a slice if and only
if it is a base of a generic real pencil of projective hypersurfaces (and,
in particular, it is a regular complete intersection of two hypersurfaces
of the same degree). 
\end{Th}

\subsection{Digression on projective hypersurfaces and pencils.}\label{s2.1}
Recall that a hypersurface in an $n$-dimensional projective space 
over a field $\mathbb F$ is defined by a single equation $P(x_0,\dots,x_n)=0$ 
where $P$ is nonzero homogeneous polynomial in $n+1$ variables with
coefficients in $\mathbb F$. Thus the hypersurfaces are parametrized by 
the polynomials considered up to a nonzero factor. In other words, the
hypersurfaces form an $n$-dimensional projective space over $\mathbb F$.

A projective line in this space is called a {\em pencil\/} of hypersurfaces
of degree $d$ in $\mathbb F P^n$. As a projective line, a pencil is defined by
any two of its points, that is by any two projective hypersurfaces over
$\mathbb F$ 
of the same degree.

Let $Y$ and $Z$ be {\em real\/} algebraic hypersurfaces of degree $d$ in the 
projective space of dimension $n$. 
As projective hypersurfaces, $Y$ and $Z$ are defined by a single polynomial 
equation each.  Let $Y$ be defined by an equation $R=0$ and $Z$
be defined by an equation $S=0$, where $R$ and $S$ are real homogeneous 
polynomials of degree $d$. The pencil containing $Y$ and $Z$ consists of
hypersurfaces defined by equations 
$$\Gl R(x_0,\dots,x_n)+\mu S(x_0,\dots,x_n)=0$$ with $(\Gl:\mu)\in \R P^1$.

For a generic pencil, the hypersurfaces can be chosen to be nonsingular and
transverse. Their intersection is nonsingular subvariety of a codimension
two in the projective space. The intersection is called the {\em base\/} of the pencil.  
The pencil consists of all irreducible hypersurfaces containing the base.

\subsection{Proof of Theorem \ref{ThPen}.  Sufficiency.}\label{s2.2}
Let $X$ be a real projective variety of codimension two which is a base of 
a generic real pencil of projective hypersurfaces of degree $d$. 
 Let us choose two generic hypersurfaces $Y$ and $Z$ belonging to the
pencil. Thus $Y$ and $Z$ are nonsingular real algebraic 
hypersurfaces of degree $d$ in the $n$-dimensional projective space. Let 
$Y_\R$ and $Z_\R$ intersect transversally. Then $X=Y\cap Z$ is a variety 
of codimension two.  Notice that $X$ is nonsingular. 

As projective hypersurfaces, $Y$ and $Z$ are defined by a single real 
polynomial 
equation each.  Let $Y$ be defined by an equation $R=0$ and $Z$
be defined by an equation $S=0$, where $R$ and $S$ are real homogeneous 
polynomials of degree $d$. 

Then $X=Y\cap Z$ is defined by the system 
$$\begin{cases}R(x_0,\dots,x_n)=0\\S(x_0,\dots,x_n)=0\end{cases}$$
and differential real 1-forms $d_p R$, $d_p S$ are linear independent 
at each point $p=(x_0,x_1,\dots,x_n)\in Y_\R\cap Z_\R$. 


Consider the complex projective hypersurface $V$ defined by the equation
$$R(x_0,\dots,x_n)+iS(x_0,\dots,x_n)=0.$$ 
It belongs to the pencil at $(\Gl,\mu)=(1,i)$. At each point 
$$p=(x_0,x_1,\dots,x_n)\in V\cap \rpn=\{x\in\rpn\mid R(x)=0, S(x)=0\}=
Y_\R\cap Z_\R$$ 
the complex 1-form 
$d_p(R+iS)=d_pR+id_pS$ does not vanish and annihilates $T_pV$. 
Therefore $V$ has complex codimension 1 and real codimension 2
in $\cpn$ at each $p\in\rx$. On the other hand, $\rx$ has real codimension
2 in $\rpn$. Hence $\dim_\R\cpn-\dim_\R V=\dim_\R\rpn-\dim_\R\rx$. This
equality implies transversality of $V$ and $\rpn$. \qed

\subsection{Proof of Theorem \ref{ThPen}. Necessity.}\label{s2.3} Let $X$ be a sliced real algebraic variety of 
codimension 2 in $\rpn$  
and $V\subset\cpn$ be a non-singular complex variety transverse to $\rpn$ 
such that $\rx=V\cap\rpn$. 

As a complex projective hypersurface, $V$ is defined by a single equation  
$P(x_0,\dots,x_n)=0$, where $P$ is a complex homogeneous polynomial. 
Since $\rx=V\cap\rpn$, we have $\rx=\{x\in\rpn\mid P(x)=0\}$.

By Theorem \ref{Th2} above, the polynomial $P$ {\em cannot be real\/}, because
the set $V$ of its zeros intersects $\rpn$ transversally. 
Thus $P=\Re P + i\Im P$, and both $\Re P$ and $\Im P$
are non-zero real homogeneous polynomials in $n+1$ variables. 
Put $\Re P=R$, $\Im P=S$ and denote  by $Y$ and $Z$ the hypersurfaces 
of $\rpn$ defined by equations 
$R=0$ and $S=0$, respectively. Obviously, for any $x\in\R^{n+1}$
$$P(x)=R(x)+iS(x)=0 
\iff R(x)=0 \text{ and } S(x)=0.$$ Therefore
\begin{multline*}
X_\R=\{x\in\rpn\mid P(x)=0\}\\=\{x\in\rpn\mid
R(x)=0\}\cap\{x\in\rpn\mid S(x)=0\}\\=Y_\R\cap Z_\R.
\end{multline*}
The tangent space of $V$ at  $p\in\rx$ is the annihilator of the complex
1-form $d_pP=d_pR+id_pS$. Since $V$ is a nonsingular complex hypersurface, 
its tangent space $T_pV$ at $p\in\rx$ has complex codimension one. 

Since $V$ is transverse to $\rpn$ at $p$,
$$ \dim_\R T_pV+\dim_\R \rpn-\dim_\R(T_pV\cap T_p\rpn)=\dim_\R\rpn=2n.$$
From this, by substituting expressions for the known dimensions, we get 
$$(2n-2)+n-\dim_\R(T_pV\cap T_p\rpn)=2n$$ and  
$\dim_\R(T_pV\cap T_p\rpn)=n-2$. 

On the other hand, $T_pV$ is the annihilator of the complex
1-form $d_pP$ which is equal to $d_pR+id_pS$. Therefore the real 
vector space $T_pV\cap T_p\rpn$ (which as we have seen has real codimension
2) has annihilator generated by real 1-forms $d_pR$ and $d_pS$. Therefore 
$d_pR$ and $d_pS$ are linear independent. Hence $Y_\R$ and $Z_\R$ are
nonsingular and transverse to each other at $p$. 

Thus $Y$ and $Z$ are real hypersurfaces nonsingular and transverse to each
other at each real intersection point and  $X=V\cap\rpn$ is a base of
pencil generated by them. \qed          

\subsection{Real hypersurfaces with real part having codimension two.}\label{s2.4}
It's worth mentioning that for a slice real projective variety $X$ 
of degree $d$ and codimension 2, there is a {\em real singular\/} hypersurface  
$Y$ of degree $2d$ such that $\ry=\rx$.  One can define such a 
hypersurface by equation $R^2+S^2=0$, where $R=0$ and $S=0$ are
equations of any two real hypersurfaces from the pencil with base $X$.
The hypersurface defined by equation $R^2+S^2=0$ is the union of two 
conjugate imaginary hypersurfaces from the same pencil: 
$R\pm iS=0$.

\section{Digression on orientations}\label{s3}
\subsection{Orientations of a real vector space.}\label{s3.1} 
Recall that an orientation of a finite-dimensional real vector space $X$ 
is a function $O$ on the 
set of bases of $X$ taking values in $\{\pm 1\}$ which satisfies the
following relation: 
$$O(b)O(b')=\sign\det T,$$ 
where $b=(b_1,\dots,b_n)$ and $b'=(b'_1,\dots,b'_n)$ are
bases of $X$, and $T$ is the transition matrix: $b'=bT$. Due to the
relation, an orientation is determined by its value on any basis and there are exactly two orientations on any $X$. 
Usually an orientation $O$ is identified with the set $O^{-1}(+1)$ 
of bases, on which $O$ takes value $+1$. 

This approach is quite confusing for a zero-dimensional $X$, because in 
this case there exists only one basis of $X$, the one that contains no 
vectors. 
Still, on the 0-dimensional space, there are two orientations: one that takes 
value $+1$ on this empty basis, and another, that takes value $-1$ on it. 
These orientations are identified with $+1$ and $-1$, respectively.    

On the contrary, for a vector space of dimension $>0$, there is no way to 
tell the difference between the two orientations, because there are 
linear automorphisms exchanging the orientations. 

A linear isomorphism $f:X\to Y$ acts contravariantly on orientations: 
an orientation $O$ of $Y$ is sent to the orientation
$f^*(O):(b_1,\dots,b_n)\mapsto O(f(b_1),\dots,f(b_n))$ of $X$. 

A distinguished orientations appear if the space has an extra structure
(for example, if it has a complex structure, see Section \ref{s3.2}), 
or if it is involved into a configuration of spaces in which the other spaces 
are oriented, see Section \ref{s3.4}.   

\subsection{Natural orientations of a complex vector space.}\label{s3.2}
Let $X$ be a vector space over $\C$. It can be considered as a vector space
over $\R$ (by forgetting multiplication of vectors by non-real complex
numbers). For each basis $b=(e_1,\dots,e_n)$ of $X$ over $\C$, one can cook 
up a basis over $\R$. There are many prescriptions to accomplish this. 
The most popular two of them give $b_\R=(e_1,ie_1,\dots,e_n,ie_n)$ and
$b'_\R=(e_1,\dots,e_n,ie_1,\dots,ie_n)$.  

Observe, that for $X=\C^2$ and $b=((1,0),(0,1))$ the bases 
$$b_\R=((1,0),(i,0),(0,1),(0,i)) \text{ \ and \ } b'_\R=((1,0),(0,1),(i,0),(0,i))$$
define different orientations. Nonetheless, application of the same 
prescriptions to all complex bases of the same complex vector space give 
bases  belonging
to the same orientation of the real space. (If the transition matrix
between two bases over $\C$ has determinant $D\in\C$, then the transition 
matrix for the corresponding real bases has determinant $D\overline
D=|D|^2>0$.)

For the purposes of this paper, we have to make a choice.
We prefer the first prescription, and call an orientation of a complex vector
space $X$ the {\em complex orientation\/} if it takes value $+1$ on 
$b_\R$. 

\subsection{Complex orientations and isomorphisms.}\label{s3.3}
Of course, a complex linear isomorphism $X\to Y$ between complex vector 
spaces maps the complex orientation of $Y$ to the complex orientation of
$X$.

What about complex {\em semilinear} isomorphisms?

Recall that a {\em semilinear\/} map $f:X\to Y$ between complex vector 
spaces $X$ and $Y$ is an additive homomorphism (i.e., $f(v+w)=f(v)+f(w)$
for any $v,w\in X$) such that $f(zv)=\overline zf(v)$ for any $v\in X$ and
$z\in\C$.  

\begin{Th}\label{lem1} 
Let $X,Y$ be complex vector spaces. Then a semilinear 
isomorphism $f:X\to Y$ maps the complex orientation of $Y$ to 
$(-1)^{\dim_\C Y}$ times the complex orientation of $X$.
\end{Th}

\begin{proof}
A semilinear isomorphism $f:X\to Y$ maps a basis $e_1,\dots,e_k$ 
of $X$ over $\C$ to a basis $f(e_1),\dots,f(e_k)$ of $Y$ over $\C$, while 
the corresponding basis $e_1,ie_1,\dots,e_k,ie_k$ over $\R$ is mapped to
$f(e_1),-if(e_1),\dots,f(e_k),-if(e_k)$ of $Y$. Let $O_X$ and 
$O_Y$ be the complex orientations of $X$ and $Y$, respectively. 
Then  $O_X(e_1,ie_1,\dots,e_k,ie_k)=1$, while
\begin{multline*}
f^*O_Y(e_1,ie_1,\dots,e_k,ie_k)\\=O_Y(f(e_1),-if(e_1),\dots,f(e_k),-if(e_k))
\\=(-1)^kO_Y(f(e_1),if(e_1),\dots,f(e_k),if(e_k))=(-1)^k
\end{multline*} 
\end{proof}

\subsection{Orientations in a short exact sequence.}\label{s3.4}
In a short exact sequence
\begin{equation}\label{shes} 
0\to A\xrightarrow{\Ga}B\xrightarrow{\Gb} C\to0
\end{equation} 
of vector spaces 
over $\R$, there are a number of ways to determine orientation of any 
of the three spaces from the orientations of the other two.
Unfortunately, there are several ways. For our purposes, we  
choose one. 

The relation among the orientations comes from relation among bases.
Let $a=(a_1,\dots,a_p)$ be a basis of $A$, let $b=(b_1,\dots,b_{p+q})$ be basis of $B$
and $c=(c_1,\dots,c_q)$ be a basis of $C$. These bases are said to be {\em
related \/} by sequence \eqref{shes}, if  $\Gb(b_i)=c_{i}$ for $i=1,\dots,q$
and $\Ga(a_j)=b_{j+q}$ for $j=1,\dots,p$.   
Orientations $O_A$, $O_B$, $O_C$ will be 
called {\em related\/} by sequence \eqref{shes}
if
\begin{multline*}O(a)O(b)O(c)\\=
O_A(a_1,\dots,a_p)O_B(b_1,\dots,b_q,\Ga(a_1),\dots,\Ga(a_p))O_C(\Gb(b_1),\dots,\Gb(b_q))\\=1\end{multline*}
for bases $a$, $b$ and $c$ related by sequence \eqref{shes}.

This rule may seem to be
counter-intuitive: why not to start the basis of $B$ starting with
the images of $a_i$. Our choice is forced, in particular, by the commonly
accepted rule about the orientation of the boundary of an oriented smooth
manifold
: according to this rule the basis of the tangent space $T_pM$ 
formed by an outwards vector followed by positively
oriented basis of the boundary $T_p\p M$ is positively oriented. Here
$A=T_p\p M$, $B=T_p M$ and $C=\fac{B}A$.

\section{The transverse bundle of a slice}\label{s4}

\subsection{Transverse bundle.}\label{s4.1}
If $V$ is a smooth submanifold of a smooth manifold $M$ and $p\in S$, then
a {\em transverse space\/} $N_p^MV$  of $V$ at $p$ is defined as 
the quotient space 
$\fac{T_pM}{T_pV}$. The fiber bundle 
formed by transverse spaces at all $p\in V$ is called the {\em transverse
bundle\/} of $V$ in $M$ and denoted by $N^MV$.
  
If $T_pM$ is equipped with an inner product (say, if $M$ is a Riemannian 
or Hermitian manifold), then there is a canonical subspace of $T_pM$ 
isomorphic to $\fac{T_pM}{T_pV}$ (the orthogonal complement of 
$T_pV$). This subspace is called the {\em normal space\/} at $p$. 
The fiber bundle formed by normal spaces at all $p\in V$ is called {\em normal
bundle\/} of $V$ in $M$. It is isomorphic to the transverse bundle.

If $M$ and $V$ are complex manifolds, then the tangent spaces $T_pM$ and
$T_pV$ are complex vector spaces. Hence, $N_p^MV=\fac{T_pM}{T_pV}$ is a
complex vector space, too.

\subsection{Complex structure on the transverse bundle of a real slice.}\label{s4.2}
\begin{Th}\label{Th1} 
 Let  $X$ be a non-singular real algebraic variety  
and $Y$ be a slice subvariety of $X$. Then the transverse bundle 
$N^{\rx}\ry$ of $\ry$ in $\rx$
admits a complex structure. The complex
structure is defined by a choice complex subvariety $V$ of $\cx$, 
wnich carves $\ry$ on $\rx$ (that is $\ry=V\cap\rx$).
\end{Th}

\begin{proof} Let us fix a point $p\in\ry$. 
Since $\ry$ is a transverse intersection of $\rx$ and $V$, 
$$T_p{\ry}=T_p{\rx}\cap T_pV.$$
Therefore, 
$$N_p^{\rx}{\ry}=\fac{T_p{\rx}}{T_p{\ry}}=\fac{T_p{\rx}}{T_p{\rx}\cap T_pV}$$
There is a canonical isomorphism 
$$\fac{T_p{\rx}}{T_p V\cap T_p{\rx}}\to\fac{(T_p V+T_p{\rx})}{T_p V}$$
Further, due to transversality of $V$ and $\rx$ at $p$, 
$$T_p V+T_p{\rx}=T_p\cx.$$
Thus, we have a canonical isomorphism 
$$N_p^{\rx}\ry\to\fac{T_p{\cx}}{T_pV}=N_p^{\cx} V$$
of the transverse space of $\ry$ in $\rx$ onto the transverse space of $V$ in
$\cx$. The latter is a complex vector space, as a transverse space of a 
complex variety in a complex variety. The isomorphism 
defines a complex structure on the former space $N_p^{\rx}{\ry}$ (which a
priori has no a complex structure, because neither $\rx$, nor $\ry$ have).
\end{proof}

\subsection{Coorientations of a slice variety.}\label{s4.3}
A submanifold with oriented transverse bundle is called {\em cooriented}. 
An orientation of the transverse bundle is called a {\em coorientation} 
of the submanifold.

\begin{corTh}
\label{corTh1} Any slice  is {\em coorientable\/}. The coorientation is
determined by a complex subvariety which carves out the slice.\qed
\end{corTh}
 
\subsection{Behavior of a slice coorientation under the complex 
conjugation.}\label{s4.4}
Coorientation of a slice depends on a complex variety which carves it.
In particular, the coorientation reverses if the codimension of slice is
not divisible by four and the complex variety is replaced by its image
under the complex conjugation. 
 
\begin{Th}\label{Th3} Under the assumptions of Theorem \ref{Th1}, the 
varieties $V$ and $\conj V$ define the same  
orientation of $N^{\rx}\ry$, if $\dim Y\equiv \dim X\pmod4$, and opposite
orientations, if 
$\dim Y\equiv \dim X+2\pmod4$.\end{Th}

\begin{proof} The isomorphisms which bring the complex structures 
from $ N_p^{\cx}V$ and $N_p^{\cx}\conj V$ to $N_p^{\rx}{\ry}$ are included
into the following commutative diagram:

$$
\begin{CD}
N_p^{\rx}{\ry}=\fac{T_p\rx}{T_p\rx\cap T_pV}@>>>\fac{T_p\cx}{T_pV}=N_p^{\cx}V \\
 @| @VV\conj_*V \\
N_p^{\rx}{\ry}=\fac{T_p\rx}{T_p\rx\cap T_p\conj V}@>>>\fac{T_p\cx}{T_p\conj
V}=N_p^{\cx}\conj V
\end{CD}
$$

The complex conjugation induces {\em semilinear\/} isomorphism
$\fac{T_p\cx}{T_pV}\to\fac{T_p\cx}{T_p\conj V}$. 

The dimension of $N_p^{\rx}{\ry}$ over $\C$ is $\frac12(\dim X-\dim Y)$.
Therefore by Theorem \ref{lem1} the orientations of $N_p^{\rx}{\ry}$,
 which are defined by the
complex orientations of $N_p^{\cx}V$ and $N_p^{\cx}\conj V$, differ by
factor $(-1)^{(\dim X-\dim Y)/2}$.
\end{proof}

\section{Coorientations and integer cohomology}\label{s5}

\subsection{A cooriented submanifold as an integer cocycle.}\label{s5.1}
Cooriented submanifolds represent
integer cohomology classes in a very handy way. Let us start with two 
similar statements:

$\bullet$ An {\em oriented\/} closed $k$-dimensional submanifold $Y\subset X$ 
represents a  {\em homology\/} class $\in H_k(X;\Z)$ of dimension $k$
with integer coefficients.

$\bullet$ A {\em cooriented\/} closed submanifold $Y\subset X$ of {\em
codimension\/} $k$ represents a {\em cohomology\/} class $\in
H^k(X;\Z)$ of dimension $k$ with integer coefficients. 

The former is well known. The latter is known, but not that well, and 
deserves to be explicated, see bellow \ref{s5.2}.

If the {\em ambient\/} manifold $X$ is {\em oriented\/}, then {\em
orientations\/} and
{\em coorientations} of a submanifold $Y\subset X$ {\em determine\/} 
each other. Then orientability and coorientability are equivalent. 
The homology and cohomology classes realized by the same submanifold with
the orientation and coorientation determined by each other are
{\em Poincar\'e dual}. 

In a {\em non-orientable\/} manifold, the classes of orientable and 
coorientable  submanifolds are {\em detached\/} from each other. 

Many important real algebraic varieties (e.g., real projective 
spaces of even dimensions) are non-orientable.
Under the assumptions of Theorem \ref{Th1}, the manifolds $\rx$ and $\ry$ 
may be orientable or not. In the simplest example,
when neither of them is orientable, $\rx=\R P^4$ and $\ry=\rpp$.
In this example, $\ry$ is {\em coorientable\/}, as being a {\em slice.}

\subsection{Formal construction of the cohomology class of a cooriented 
closed submanifold.}\label{s5.2}
Let $X$ be a smooth closed manifold of dimension $n$ and $Y$ be its smooth 
closed submanifold of codimension $k$. Let $T$ be a tubular neighborhood of
$Y$ in $X$. 

A coorientation of $Y$ defines orientations on all fibers of  $T$ and 
also the Thom class 
$u\in H^k(T,\p T;\Z)=H^{k}(X,X\sminus\Int T;\Z)$, which
takes value $+1$ on the orientation class of each fiber of $T$. 
The integer cohomology class realized by $Y$ is the image of the Thom class
$u$ 
under the relativization homomorphism
$$H^{k}(X,X\sminus\Int T;\Z)\to H^{k}(X;\Z).$$ 

\subsection{Relative and compact support cohomology classes.}
\label{s5.3}
The construction of Section \ref{s5.2} admits natural extensions and 
variations. First of all, like cochains in the simplicial theory, cooriented
submanifolds can be equipped with coefficients (say, rational numbers). This
turns a cooriented submanifold to a representative of a cohomology class 
with the corresponding coefficient coefficient group.

Further, submaifolds may be non-compact or have boundary and represent
cohomology classes in manifolds with boundary or non-compact manifolds.  

In a smooth compact manifold $X$ {\em with boundary}, an {\em absolute 
cohomology\/} class with integer coefficients is realized by a cooriented 
submanifold 
$Y$ {\em with boundary} $\p Y=Y\cap \p X$. The construction is literally
the same as in Section \ref{s5.2}. 

If a cooriented submanifold $Y$ is {\em closed\/}, then the homomorphism\\
$H^{k}(X,X\sminus\Int T;\Z)\to H^{k}(X;\Z)$ can be replaced by 
$H^{k}(X,X\sminus\Int T;\Z)\to H^{k}(X,\p X;\Z)$. This gives a realization of 
a {\em relative cohomology\/} class. Namely, $Y$ realizes the image of the 
Thom class $u\in H^k(X,\p X;\Z)$ under the inclusion homomorphism
$H^{k}(X,X\sminus\Int T;\Z)\to H^{k}(X,\p X;\Z)$.

If the ambient manifold $X$ is {\em non-compact\/}, then a cooriented
{\em compact submanifold\/} would realize a cohomology class with 
{\em compact support}. 
For realization of a cohomology class with (usual) {\em closed support}, one 
may need to choose a {\em non-compact\/} cooriented submanifold, which is
closed as a subset of $X$. 

\subsection{Functoriality.}\label{s5.4}
Let $M$ and $X$ be smooth manifolds and $f:M\to X$ be a differentiable map.
Let $Y\subset X$ be a cooriented submanifold of codimension $k$ and
$\eta\in H^k(X;\Z)$ be a cohomology class realised by $Y$. 
 
If $f$ is transverse to $Y$, then $N=f^{-1}Y$ is a smooth submanifold of 
$M$, which has the same codimension $k$, and $f$ induces a fiberwise 
isomorphism $N_MN\to N_XY$. Then the cohomology class realised by 
$N$ in $M$ is the image of $\eta$ under $f^*:H^k(X;\Z)\to H^k(M;\Z)$.

\subsection{Intersection number of oriented and cooriented
submanifolds in a non-oriented manifold.}\label{s5.5}
Let $X$ be a ({\em non-oriented}) smooth manifold, let $Y$ be an  
{\em oriented\/} smooth $k$-dimensional
submanifold of $X$ and let $Z$ be a {\em cooriented\/} smooth 
submanifold of $X$ of codimension $k$, {\em transverse\/} to $Y$. Since $\dim
Y=\codim Z$, the intersection $Y\cap Z$ consists of isolated points.
Assume that the intersection is finite (this is automatically true if $Y$
and $Z$ are compact). Define the {\em intersection number\/} $Y\circ Z$ 
as the sum $\sum_{p\in Y\cap Z}Y\circ_pZ$ of all 
{\em local intersection numbers\/} $Y\circ_pZ$ which are defined as follows.  

At $p\in Y\cap Z$ take 
$T_pY\subset T_pX$ and map it to $N_p^XZ=\fac{T_pX}{T_pZ}$ by the
restriction of the natural projection $T_pX\to\fac{T_pX}{T_pZ}$. The tangent
space $T_pY$ is oriented (by the orientation of $Y$), the transverse space 
$N_p^XZ$ is oriented (by the coorientation of $Z$). 
The map is an isomorphism due to transversality of $Y$ and $Z$. 
If this isomorphism maps the orientation of $N_p^XZ$ to the orientation 
of $T_pY$, then $Y\circ_pZ$, the {\em local intersection
number at\/} $p$, is defined to be $+1$, otherwise it is 
defined to be $-1$.

Usually intersection number is defined for transverse {\em oriented\/} 
submanifolds of complementary dimensions in an {\em oriented\/} manifold. 
The definition above
is more general. Indeed, if the ambient manifold is oriented, 
then the orientation
of one of the submanifolds is canonically converted to a 
coorientation and the
definition above gives the usual intersection number. 

If the manifolds $X$, $Y$ and $Z$ are closed, then the intersection number of
$Y$ and $Z$ equals the value of the cohomology class $z\in H^k(X;\Z)$, which
is realized by $Z$, on the homology class $y\in H_k(X;\Z)$, which is 
realized by $Y$.

\subsection{Products.}\label{s5.6}
Let $X$ be a smooth (non-oriented) manifold. 

\noindent
{\bf Cup-product.} If classes $\eta\in H^k(X;\Z)$ and $\Gz\in H^l(X;\Z)$ 
are realized by transverse cooriented closed submanifolds $Y$ and $Z$, then
$\eta\smallsmile\Gz\in H^{k+l}(X;\Z)$ is realized by $Y\cap Z$ with
coorientation obtained via an isomorphism between the transverse bundle
$N^X(Y\cap Z)$ and the direct sum of the restrictions of the transverse 
bundles of $Y$ and $Z$ to $Y\cap Z$. Over any $p\in Y\cap Z$ the 
isomorphism acts as follows:
\begin{multline*}N_p^X(Y\cap Z)=
\fac{T_pX}{T_pY\cap T_pZ}=\fac{(T_pY+T_pZ)}{T_pY\cap T_pZ}\\
=\fac{T_pY}{T_pY\cap T_pZ}\oplus\fac{T_pZ}{T_pY\cap T_pZ}\\
=\fac{(T_pY+T_pZ)}{T_pZ}\oplus\fac{(T_pY+T_pZ)}{T_pY}\\
=\fac{T_pX}{T_pZ}\oplus\fac{T_pX}{T_pY}
=N_p^X(Z)\oplus N_p^X(Y)\end{multline*}
 
\noindent
{\bf Cap-product.} If classes $\eta\in H^n(X;\Z)$ and $\Gz\in H^k(X;\Z)$ 
are realized by by transverse oriented and cooriented closed submanifolds  
$Y$ and $Z$, then $\eta\smallfrown\Gz\in H^{n-k}(X;\Z)$ is realized by 
$Y\cap Z$ oriented from $$0\to T(Y\cap Z)\to TY|_{Y\cap Z}\to
\fac{(TY|_{Y\cap Z})}{T(Y\cap Z)}\to 0,$$ since for any $p\in Y\cap Z$ 
$$\fac{T_pY}{T_pY\cap
T_pZ}=\fac{(T_pY+T_pZ)}{T_pZ}=\fac{T_pX}{T_pZ}=N_p^XZ.$$
Evaluation of an integer cohomology class on an
integer homology class of the same dimension, considered above in Section
\ref{s5.5}, is a special case of the cap-product.

\subsection{Would be good in a textbook.}\label{s5.7}
Representation of an integer cohomology class by cooriented submanifolds 
can hardly be found in Algebraic Topology textbooks. Though, it is mentioned
in Wikipedia's article on cohomology. 

Modern expositions of homology theory tend to expel 
geometric objects, like cycles and cocycles, from the major 
theorems and hide them under the hood of technical definitions and proofs.

I believe that the material of this section would be highly appropriate 
in a textbook. 
It shows a 
geometric part of the story, demonstrates that cohomology classes are 
not less geometric than homology classes and disclose a nice geometry 
of operations with homology and cohomology.

As tools for illustrating homology theory,
oriented and cooriented submanifolds may compete with integer cycles and
cocycles. Submanifolds provide more transparent illustration
(thanks to usage of transversality and pullback construction). 
However, simplicial cycles and cocycles have two well-known 
advantages. 

{\em Firstly}, the class of simplicial spaces seems to be broader 
than the class of smooth manifolds. 

{\em Secondly}, there exist homology and cohomology classes of a 
smooth manifold, 
which cannot be realized by smooth oriented and cooriented submanifolds.
This was discovered by Thom \cite{Th}. 

The problem solved by Thom was posed as realizability of {\em
homology\/} integer classes by oriented submanifolds. However, 
technically Thom considered realization of a {\em cohomology\/} 
class Poincar\'e dual to the homology one. In particular, he 
considered realization of integer homology classes only in {\em
orientable\/} manifolds. 

As long as we speak about slices, both issues are {\em irrelevant\/}: we study
non-singular real algebraic varieties, which are smooth manifolds, slices 
are smooth cooriented submanifolds. 

From the viewpoint of homotopy theory the class of compact
finite-dimensional simplicial spaces is not really broader than the class
of compact smooth manifolds. A compact 
finite-dimensional simplicial space can be embedded to a smoothly 
triangulated $\R^n$ in such a way that its 
regular neighborhood would be a smooth manifold homotopy 
equivalent to the original space. 

Furthermore, many homology and cohomology classes with integer 
coefficients are realizable. 
In particular, all integer cohomology classes in a closed smooth manifold 
of dimension $\le 9$ are realizable. All integer cohomology classes of 
dimensions 1 and 2 are realizable. For any integer cohomology class $u$,
some its non-zero multiple $Nu$ is realizable.  

The realizability problem disappears if 
closed smooth submanifolds are replaced by closed 
submanifolds with singularities of codimension $\ge2$ (i.e., by a passage
from submanifolds to pseudomanifolds), or by chains made of 
smooth compact submanifolds with boundaries and corners. 

However, writing a textbook is not the purpose of this paper. For what
follows, we restrict ourselves on one application to topology of real 
algebraic varieties. It is based on the classical construction of linking
coefficients of a cycle and cocyle, see, e.g. Lefschetz \cite{Lef},
interpreted in terms of oriented and cooriented submanifolds. In the next
section a generalization of linking coefficients is sketched.  
We skip routine verifications. 

\subsection{Cap-linking of oriented and cooriented
submanifolds.}\label{s5.8} 
Let $X$ be a smooth closed manifold. Let $A$ be its
smooth closed {\bf oriented} submanifold of dimension $n$ and $B$ be a smooth
closed {\bf cooriented} submanifold of $X$ of codimension $k\le n+1$.
 Assume that $A\cap B=\empt$. 

Assume that the homology class $\Ga\in H_n(X)$ realized by $A$ and the
integer cohomology class $\Gb\in H^k(X;\Z)$  realized by $B$ are of finite 
orders: $p\Ga=0$ and $q\Gb=0$ for some integer $p$ and $q$. 

Under these assumptions, $A$ and $B$ define a
homology class $\lk_\frown(A,B)\in H_{n+1-k}(X;\Q)$. We will call this
class {\em cap-linking} of $A$ and $B$. It depends not 
only of $\Ga$ and $\Gb$, but of $A$ and $B$, or, rather, on the isotopy class
of $A\cup B\subset X$. When $k=n+1$ and $X$ is oriented, then
$\lk_{\frown}(A,B)\in H_0(X;\Q)$ is the linking number of $A$ and $B$
equipped with orientation dual to the initial coorientation of $B$.

The cap-linking  of $\lk_{\frown}(A,B)$ can be obtained by two constructions, which give the same result. 

In one of them, we choose a smooth compact {\bf oriented} submanifold  $A'$ 
transverse to $B$ with $\p A'=pA$ and take $\frac1pA'\frown B$.

In the other one, we choose a smooth compact {\bf cooriented} submanifold  
$B'$ transverse to $A$ with $\p B'=qB$ and take $A\frown \frac1q B'$.

Although $A'$ and $B'$ have non-empty boundary, 
both intersections, $A'\cap B$ and $A\cap B'$, have empty
boundary and are oriented closed $n+1-k$ dimensional submanifolds. 
Indeed, since $A\cap B=\empt$, \vspace{-7pt}
\begin{itemize}\setlength\itemsep{-4pt}
\setlength\topsep{-4pt}
\setlength\parsep{0pt}
\item $\p (A'\cap B)=\p A'\cap B=pA\cap B=p(A\cap B)=\empt$ and 
\item $\p(A\cap B')=A\cap\p B'=A\cap qB=q(A\cap B)=\empt$. 
\end{itemize}\vspace{-7pt}
As rational multiples of intersections of oriented closed submanifolds and 
cooriented closed submanifolds,  $\frac1pA'\cap B$ and $A\cap \frac1qB'$ 
are rational multiples of oriented closed submanifolds which define 
elements of $H_{n+1-k}(X;\Q)$. 

In fact, this is the same element.  
In order to prove that, make $A'$ and $B'$ 
transverse to each other. It can be done without changing them near their
boundaries $\p A'=\frac1pA$ and $\p B'=\frac1q B$ since $A\cap B=\empt$. 

Consider the intersection $A'\cap B'$. This is a smooth manifold. By
Leibniz rule, 
$$\p(A'\cap B')=(\p A'\cap B')\cup(A'\cap \p B')=(p A\cap B')\cup(A'\cap q
B)$$
Hence,
$$\p\tfrac1{pq}(A'\frown B')=\tfrac1{pq}\left((p A\frown B')\cup(A'\frown q
B)\right)=(A\frown \tfrac1qB')\cup (\tfrac1pA'\frown B). $$

 In fact, $A\frown \tfrac1qB'$ and $\tfrac1pA'\frown B$
are involved in $\p\tfrac1{pq}(A'\frown B')$ with opposite orientations, 
cf. the formula for the boundary of cap-product for $y\in C_n$ and $z\in
C^k$
$$\p(y\frown z)=(-1)^k((\p y)\frown z-y\frown \p z) $$
in the usual simplicial complexes.

\subsection{Linking number of a point and a curve on $\R P^2$ }\label{s5.9}
Consider the simplest example of the linking number. The ambient
space is the real projective plane $\R P^2$. Cf. \cite{Vir}, where the
linking number was used in the statement of Rokhlin's complex orientation
formula. There it appeared under the name of {\em index of a point with
respect to a curve}. 

On the projective plane, a single oriented point does not bound, because
$H_0(\rpp)=\Z$.

A double co-oriented point
bounds a projective line passing through it, with a co-orientation, 
which flips at the point.

\centerline{\includegraphics{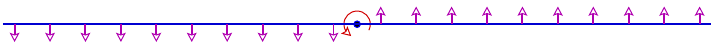}}

This allows to consider the linking number of a double co-oriented 
point with any oriented closed curve.

\centerline{\begin{picture}(0,0)%
\includegraphics{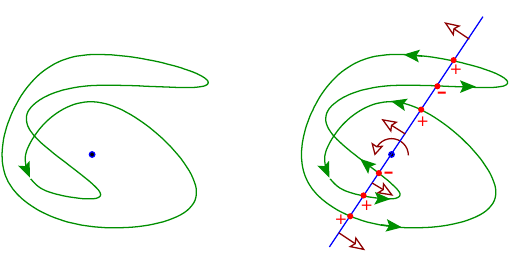}%
\end{picture}%
\setlength{\unitlength}{5387sp}%
\begin{picture}(3012,1532)(-1390,-332)
\put(939,229){\makebox(0,0)[lb]{\smash{\fontsize{9}{10.8}\normalfont {\color[rgb]{0,0,1}$A$}%
}}}
\put(-816,229){\makebox(0,0)[lb]{\smash{\fontsize{9}{10.8}\normalfont {\color[rgb]{0,0,1}$A$}%
}}}
\put(1496,-94){\makebox(0,0)[lb]{\smash{\fontsize{9}{10.8}\normalfont {\color[rgb]{0,0,1}$C$}%
}}}
\put(-269,-106){\makebox(0,0)[lb]{\smash{\fontsize{9}{10.8}\normalfont {\color[rgb]{0,0,1}$C$}%
}}}
\end{picture}%
}

$\lk_{\frown}(C,2A)=2$, $\lk_{\frown}(C,A)=1$. By a similar calculation one
can show 
that $\lk_{\frown}(\R P^1,A)=\frac12$.

\section{Linking of real algebraic varieties}\label{s6}
In this section we come back to real algebraic geometry. 

\subsection{Linking of dividing curves and codimension 2 slice varieties.}\label{s6.1}
Let $n$ be an integer $>1$, let $A$ be a non-singular real algebraic projective curve of type I in 
$\rpn$.  Let $\ca$ realize $d\,[\cpl]\in H_2(\cpn)$.

Let $B$ be a slice real algebraic variety of dimension
$n-2$ of the same space $\rpn$.  Let $V\subset\cpn$ be a non-singular 
complex hypersurface transverse to $\rpn$ with $\rb=V\cap \rpn$, let
$[V]=D\,[\C
P^{n-1}]\in H_{2n-2}(\cpn)$ and $\ra\cap\rb=\empt$.

As a curve of type I, $\ra$ divides $\ca$ into two halves. Let $H$ be one
of the halves. Orient $\ra$ as $\p H$. 

As a slice variety, $\rb$ is cooriented. Since $\ra\cap\rb=\empt$ and 
$H_1(\rpn)=\fac\Z2$, there is a well defined cap-linking number 
$\lk_\frown(\ra,\rb)\in\frac12\Z$. 

\begin{Th}\label{Th5} Under the assumption above, $\frac12Dd=H\circ
V+\lk_\frown(\ra,\rb)$. 
\end{Th}\medskip

\begin{proof}
Since $H_1(\rpn)=\fac\Z2$, there exists a smooth oriented compact surface 
$A'\subset \rpn$ with $\p A'\subset\ra$ and 
$$[\p A']=
2[\ra]\in H_1(\ra) $$
Assume that $A'$ is transverse to $\rb$. Then
$\lk_{\frown}(\ra,\rb)=\frac12A'\frown\rb$, where $A'\frown\rb$ is defined in
Section \ref{s5.6}, which is the same as the
intersection number of $A'$ and $\rb$ in Section \ref{s5.5}.

\begin{lem}\label{lem2} 
Form an integer cycle $R=H\cup -\frac12A'$. The homology class $[R]$ realized by $R$ is 
$\frac12d\,[\C P^1]$. 
\end{lem}

\begin{proof} The following arguments are borrowed from Rokhlin's proof of 
the complex orientation formula, see the end of section 2 in \cite{Rokh}. 
Observe that $\conj[\C P^1]=-[\C P^1]$, cf. Theorem \ref{lem1}.
The group $H_2(\cpn)$ is generated by $[\C P^1]$.
Therefore  $\conj_*:H_2(\cpn)\to H_2(\cpn)$ acts as multiplication by 
$-1$. Hence $\conj_*[R]=-[R]$ and $[R]-\conj[R]=2[R]$. 

The cycle $\conj R$ is $\conj H-\frac12\conj A'=\conj H-\frac12A'$.
Therefore  $[R]-\conj_*[R]=[H-\conj H]=[\ca]=d\,[\C P^1]$.
Thus $[R]=\frac12d\,[\C P^1]$.
\end{proof} 

The rest of the proof of Theorem \ref{Th5} is based on two evaluations
of the intersection number $R\circ V$. On one hand, since $R$ realizes 
$\frac12d\,[\C P^1]$ and $V$ realizes $D[\C P^{n-1}]$, their intersection
number equals $\frac12dD$.

On the other hand, 
$$R\circ V=(H\cup -\frac12 A')\circ V=H\circ V-\frac12 A'\circ V$$
Here the $\circ$ denotes the intersection number of transverse oriented 
submanifolds of complementary dimensions in $\cpn$.
\end{proof}

\subsection{Purely real Corollary of Theorem \ref{Th5}.}\label{s6.2} 
$|\lk_{\frown}(\ra,\rb)|\le \frac12 Dd$.\medskip

\noindent
{\bf Derivation of Corollary from Theorem \ref{Th5}.} Obviously, 
$$0\le H\circ V\le\ca\circ V=Dd.$$ 
By Theorem \ref{Th5}, $H\circ V=\frac12 Dd-\lk_{\frown}(\ra,\rb)$. 
Hence, $$0\le\frac12 Dd-\lk_{\frown}(\ra,\rb)\le Dd .$$
By subtracting $\frac12 Dd$, we get
$$-\frac12Dd\le-\lk_{\frown}(\ra,\rb)\le \frac12Dd ,$$
 i.e., $|\lk_\frown(\ra,\rb)|\le \frac12 Dd$.\qed \medskip

\subsection{On the projective plane.}\label{s6.3}
On $\rpp$, a slice variety $B$ of codimension two is 0-dimensional. 
It is a base for a real pencil $\{C_t\}_{t\in\R P^0}$ of real curves.

How to perceive a coorientation of $B_\R$? The real curves $C_t$, which 
belong to the pencil, pass through the points of $B_\R$. While moving 
in the pencil, $C_t$ rotates around the base points. The 
coorientation at a point of $B_\R$ is the direction of rotation. 

Along a curve which belongs to the pencil, the coorientations alternate. On
the whole projective plane, it is impossible to compare coorientations of 
different points, because the projective plane is non-orientable. 

A choice of an oriented real line $A_\R$ disjoint from $B_\R$ turns $rpp$
to $\R^2$ and makes a comparison of coorientations possible. Moreover, 
a choice of orientation of $A_\R$ determines the sign of a coorientation
at each point of $B_\R$. The total sum of the signs over the whole $B_\R$
is $\lk_{\frown}(\ra,\rb)$. Corollary of Theorem \ref{Th5} gives a strong
upper bound for the difference between the numbers of points in $B_\R$ with
different coorientations.


\begin{thebibliography}{99999}

\bibitem{HW}  P.~J.~Hilton and S.~Wylie, {\it Homology Theory, an 
Introduction to Algebraic Topology.\/} Cambridge University Press, London, 1960.

\bibitem{Gor} R.~Mark~Goresky, {\it Whitney stratified chains and
	cochains,\/}
Trans. AMS, 267:1 (1981), 175-196.
 
\bibitem{Kh} A.~G.~Khovanskii, {\it Fewnomials.\/} Amer. Math. Soc. 
Translations of mathematical monographs; v. 88, 1991.

\bibitem{Lef} S.~Lefschetz, {\it Algebraic Topology.\/} Amer. Math. Soc. 
Coll. Publ., Vol. 27, 1942.

\bibitem{Rokh} V.~A.~Rokhlin, {\it Complex orientation of real algebraic
	curves,\/} Func. Analysis and Appl. 8:4 (1974), 71-75.

\bibitem{Th} R.~Thom, {\it Quelques propri\'et\'es globales des
vari\'et\'es diff\'erentiables,\/} Comment. Math. Helv., 28 (1954),
17-86.

\bibitem{Vir} Oleg Viro, {\it Generic immersions of circle to surfaces and complex topology
of real algebraic curves\/},  Topology of real algebraic varieties and relate
d topics, (a volume dedicated to memory of D.A.Gudkov), AMS Translations,
Series 2, {\bf 173}, (1995) 231-252.

 
\end{thebibliography}
\end{document}